\documentclass{cmslatex}
\usepackage[paperwidth=7in, paperheight=10in, margin=.875in]{geometry}
\usepackage[backref,colorlinks,linkcolor=red,anchorcolor=green,citecolor=blue]{hyperref}
\usepackage{amsfonts,amssymb}
\usepackage{amsmath}
\usepackage{graphicx}
\usepackage{cite}
\usepackage{enumerate}
\renewcommand{\theequation}{\arabic{section}.\arabic{equation}}

\allowdisplaybreaks
\def\disp{\displaystyle}
\newcommand{\R}{\mathbb{R}}

\newcommand{\dx}{\partial_x}
\newcommand{\dy}{\partial_y}

\newcommand{\va}{{\varphi}}

\newcommand{\eps}{\varepsilon}

\begin{document}
 \title{Global weak solutions to inviscid Burgers-Vlasov equations \thanks{Received date, and accepted date (The correct dates will be entered by the editor).}}

          %For each author, make a block with the following macros:

          \author{Huimin Yu\thanks{Department of Mathematics, Shandong Normal University, 250300, No.1 Daxue Road, Jinan, China, (hmyu@amss.ac.cn).}
          \and Wentao Cao \thanks{Corresponding author, Institute f\"{u}r mathematik, Universit\"{a}t Leipzig, D-04109, Augustplatz 10, Leipzig, Germany, (wentao.cao@math.uni-leipzig.de)}}

         \pagestyle{myheadings} \markboth{Burgers-Vlasov equations}{H. Yu and W. Cao} \maketitle

          \begin{abstract}
               In this paper, we consider the existence of global weak solutions to a one dimensional fluid-particles interaction model: inviscid Burgers-Vlasov equations with fluid velocity in $L^\infty$ and  particles' probability density in $L^1$. Our weak solution is also an entropy solution to inviscid Burgers' equation. The approach is adding ingeniously artificial viscosity to  construct approximate solutions satisfying $L^\infty$ compensated compactness framework and weak $L^1$ compactness framework.  It is worthy to be pointed out that the bounds of fluid velocity and the kinetic energy of particles' probability density are both independent of time.
          \end{abstract}
\begin{keywords}  weak solution; fluid-particles interaction; $L^\infty$ velocity; $L^1$ density; compensated compactness; Dunford-Pettis theroem.
\end{keywords}

 \begin{AMS} 76T10, 35F20, 35Q35, 35Q72,  45K05, 82D05.
\end{AMS}

%
%%%%%%%%%%%%%%%%%%%%%%%%%%%%%%%%%%%%%%%%%%%%%%%%%%
\section{Introduction}
%%%%%%%%%%%%%%%%%%%%%%%%%%%%%%%%%%%%%%%%%%%%%%%%%%%%%%%%%%%%%%%%%
%

We consider the following inviscid Burgers-Vlasov equations:
\begin{equation}\label{e:burgers-vlasov}
\left\{
\begin{array}{ll}
\displaystyle u_t+uu_x=\int_\R fvdv-u\int_\R fdv,  \\
\displaystyle f_t+v f_x+(f(u-v))_v=0,
\end{array}
\right.
\end{equation}
with the initial data
\begin{equation}\label{e:initial}
u(x, 0)=u_0(x), \quad f(x, v, 0)=f_0(x, v)\geq0.
\end{equation}
The system \eqref{e:burgers-vlasov} is one kind of simple model about inviscid fluid-particles interaction. The
motion of the fluid with bulk velocity $u(x, t)$ is modeled by the inviscid Burgers equation, while the dispersed particles with probability density function $f(x, v, t)$ is described by Vlasov like equation. The interaction between the fluid and the dispersed particles is achieved by a friction term between the bulk velocity of fluid and velocity of the dispersed particles, namely the drag force term $\int_\R f(v-u)dv.$

For related fluid-structure models, we shall  first mention the following diffusive system, Burgers-Vlasov equations:
\begin{equation}\label{e:model}
\left\{
\begin{array}{ll}
\displaystyle \rho_g(u_t+uu_x-\nu_g u_{xx})=E_d,  \\
\displaystyle f_t+v f_x+(F_df)_v=0,
\end{array}
\right.
\end{equation}
in which a dispersed phase interacts with a viscous gas. Here $\rho_g$ is the density of gas. The force term $E_d$ describes the exchange of impulse between the gas and particles, and the drag force $F_d$ is used to describe the friction of viscous gas on the droplets. The force terms are related by the  following formulas:
\begin{align}
&E_d=C(r)\rho_p(u_p-u), \quad F_d=C(r)(u(x, t)-v), \notag \\
&\rho_p=\frac{4\pi}{3}\rho_lr^3\int_\R f(x, v, t)dv,\quad
\rho_p u_p=\frac{4\pi}{3}\rho_lr^3\int_\R f(x, v, t)vdv. \label{e:efd}
\end{align}
In \eqref{e:efd}, $\rho_l$ is the density of the liquid and $C(r)$ is a constant depending on the radius $r$ of droplets in the spray consisting of dispersed particles. After assuming that the gas is of constant mass density $\rho_g$ and simplified the momentum equation of the gas, the Burgers' equation, i.e. the first equation in \eqref{e:model}, is utilized  to model the evolution of the viscous gas. Further assuming the spray is of enough dilution and neglecting gravity effect, a Vlasov like equation, i.e. the second equation in \eqref{e:model}, is then applied to  govern the evolution of the particles. Other detailed information about the derivations and assumptions on \eqref{e:model} can also be found in \cite{Will85, G01, DR99}. For the mathematical analysis of \eqref{e:model}, as far as we know, the first global existence and uniqueness of classical solutions to the  Cauchy problem has been considered in \cite{DR99}, in which the Burgers-Vlasov equations are equipped with regular and compact supported initial data. Meanwhile, the stability of travelling wave are also considered. When the initial data is less regular than which in \cite{DR99}, the global existence and uniqueness of finite energy solutions are proved in \cite{G01}. The second related simpler 1-D model on fluid-structure interaction
\begin{equation}\label{e:fluid-structure}
\begin{cases}
u_t+uu_x=\lambda(h'(t)-u(t, h(t)))\delta_{h(t)},\\
mh''(t)=-\lambda(h'(t)-u(t, h(t))),
\end{cases}
\end{equation}
is also considered in \cite{LST08}, where $u$ is the velocity of the inviscid fluid and  $h(t)$ is the location of the particles. $\lambda$ is the positive friction constant and $\delta_{h(t)}$ is the Dirac measure at $h(t).$ Global entropy weak solution involving shock waves to the system \eqref{e:fluid-structure} are obtained in \cite{LST08}. There are also some other fluid-kinetic models: compressible/incompressible Euler/Navier-Stokes equations coupled with Vlasov/-Fokker-Plank equations.  Weak solutions or classical solution close to the equilibrium  are studied in \cite{RS96-1, RS96-2, Y13, MV07, LMW17}.  Some asymptotic problems such as hydrodynamic limit/stratified limit of viscous Burgers-Vlasov equation and Euler/Navier-Stokes equations coupling with Vlasov equation are also considered in \cite{G01, BMP07, NPS01, GJ04-1, GJ04-2, Ja00-C, Ja00-P, MV08}.

In this paper, we investigate the existence of global weak solutions to the Cauchy problem \eqref{e:burgers-vlasov}-\eqref{e:initial}. Since the derivation of the model in \cite{Do-DCDS} is in 1D, and as far as the authors know, references on 2D or 3D cases are not founded for such model,  we consider the 1D invisid Burgers-Vlasov equations. Comparing with the diffusive system \eqref{e:model}, without viscosity term, shock wave may exist when the initial data are given arbitrarily large.  Hence we consider the entropy weak solutions to inviscid Burgers' equation. Because of the nonlocal source term in the Burgers' equation, we require the solution of the Vlasov equation  be of finite kinetic energy. Consequently, we define  $L^\infty-L^1$ weak solution to \eqref{e:burgers-vlasov}.

\begin{definition}\label{d:weak-solution}
 For any fixed  $T\in(0, \infty)$, a pair of functions $u:\R\times[0, T]\rightarrow\R,$ $f:\R^2\times[0, T]\rightarrow[0, \infty)$ is called a global \textit{$L^\infty-L^1$ weak solution} of Cauchy problem \eqref{e:burgers-vlasov}-\eqref{e:initial} if the following statements hold:
\begin{itemize}
\item[(1)] $u(x, t)\in L^\infty(\R\times[0, T])$ and $f(x, v, t)\in L^\infty([0, T], (1+v^2)L^1(\R^2)).$

\item[(2)] $u(x, t)$ is an entropy solution to Burgers' equation, i.e.
for any $\phi\in C^1_c(\R\times[0, T)),$
\begin{align}\label{e:burgersweak}
\displaystyle\int_\R\phi(x, 0)u_0(x)dx+\int_0^T\!\!\!\int_\R\left(u\phi_t+\frac{1}{2}u^2 \phi_x+\phi\int_\R f(v-u)dv\right)dxdt=0,
\end{align}
and for any convex entropy pair $(\eta, q)$ the following entropy inequality
\begin{equation}\label{e:entropy-inequality}
\eta(u)_t +q(u)_x+\eta'(u)\!\!\int_\R f(u-v)dv\leq0
\end{equation}
holds in the sense of distributions, where $(\eta, q)$ satisfies $\eta'(u) u=q'(u)$ and $\eta(u)$ is convex with respect to $u$.

\item[(3)] $f(x, v, t)$ is a weak solution to Vlasov equation, i.e. for any $\psi\in C^1_c(\R^2\times[0, T)),$
\begin{align}\label{e:vlasovweak}
\displaystyle\int_\R\int_\R\psi(x, v, 0)f_0(x, v)dxdv+\int_0^T\!\!\!\int_\R\int_\R\left(f\psi_t+vf \psi_x+\psi_v f(u-v)\right)dvdxdt=0.
\end{align}
\end{itemize}
\end{definition}
Now we are ready to state our main result.
\begin{theorem}\label{t:bounded-weak}{\bf[Main theorem]}
Let initial data $(u_0, f_0)$ satisfy
\begin{equation}\label{e:initial-condition}
\|u_0(x)\|_{L^\infty(\R)}+\|(1+v^2)f_0(x, v)\|_{L^1(\R^2)}\leq M_0
\end{equation}
for some positive constant $M_0$.  Then there exists a global $L^\infty-L^1$ weak solution to \eqref{e:burgers-vlasov}-\eqref{e:initial} in the sense of Definition \ref{d:weak-solution} and there  is a constant $M$ depending solely on $M_0$ such that
\begin{equation}
\label{e:solution-bound}
\|u(x, t)\|_{L^\infty(\R\times[0, T])}+\|(1+v^2)f(x, v, t)\|_{L^\infty([0, T], L^1(\R^2))}\leq M.
\end{equation}
\end{theorem}
Hereafter, $M$ denotes the constant depending only on $M_0$ and it may vary from line to line.
%$M_T$ denotes geometric constant depending on $M_0$ and $T$,. We also give a remark on the bounds of our solution.
\begin{remark}
 Our uniform bounds of velocity $\|u\|_{L^\infty}$ and kinetic energy of Vlasov equation $\int_\R\!\!\int_\R f(1+v^2)dvdx$  are both independent of time $T.$
\end{remark}

Our strategy of proving Theorem \ref{t:bounded-weak} is to construct approximate solutions by adding artificial viscosity to Burgers equation technically and regard the nonlocal term $\!\!\int_\R f(v-u)dv$ as a dissipative source term  in some sense. Maximum principles of parabolic equation and transport equation are applied to establish the uniform bound.  Besides, in order to get the uniform estimates ( also be independent of time $T$) of approximate viscosity solution $u^\varepsilon$ , we add some novel viscosity term and choose a special control function.  In proving almost everywhere convergence of $u^\eps,$ we employ $L^\infty$ compensated compactness framework. More information about compensated compactness framework of $L^p$ or $L^\infty$ space can be found in \cite{M78, D83-CMP, D83-ARMA,CP10, GL16} and the reference therein. On the other hand, to show the weak $L^1$ convergence of $f^\eps$, we  apply Dunford-Pettis theorem and analyze kinetic energy of $f^\varepsilon$  and evolution of sets. In the proof, we also came across the difficulty on the weak convergence of $\int vf^\eps dv$, which is overcame again by applying the uniform bound of kinetic energy.

 To be concise, in the present paper, we use $\int$ instead of $\int_\R$.  $C(\cdot)$ denotes constant depending on the parameters in the bracket. The rest of the paper is organised as follows. Section \ref{s:approximate-solution} is devoted to construct approximate solutions and prove their global existence.  The proof of Theorem \ref{t:bounded-weak} is given in Section \ref{s:proof}.

%%%%%%%%%%%%%%%%%%%%%%%%%%%%
\section{Approximate solutions}\label{s:approximate-solution}
%%%%%%%%%%%%%%%%%%%%%%%%%%%
%%%%
In this section, we construct the globally existing approximate solutions to problem \eqref{e:burgers-vlasov}-\eqref{e:initial} by adding artificial viscosity and choosing the initial data technically.
 That is, we consider the approximate problem
\begin{equation}\label{e:burgers-vlasov-viscosity}
\left\{
\begin{array}{ll}
\displaystyle u^\eps_t+u^\eps u^\eps_x=\eps u^\eps_{xx}+\eps\left(\!\int f^\eps dv\right)_x+\int f^\eps vdv-u^\eps\int f^\eps dv, \\
\displaystyle f^\eps_t+v f^\eps_x+(f^\eps(u^\eps-v))_v=0,
\end{array}
\right.
\end{equation}
and the carefully selected initial data
\begin{align}
&u^\eps(x, 0)=u_0^\eps(x)=u_0(x)*j_\eps(x), \notag\\
&f^\eps(x, v, 0)=f^\eps_0(x, v)=\left[\min\{\eps^{-{1}/{6}}, \,f_0(x, v)\mathbf{1}_{\{|x|+|v|\leq\eps^{-{1}/{6}}\}} \}\right]*j_\eps(x)*j_\eps(v)\geq0, \label{e:initial-viscosity}
\end{align}
where $j_\eps$ is standard one dimensional mollifier with parameter $\eps.$ The initial data here is chosen to make $f^\eps$ be of explicit $\eps$-depending compact support and $L^\infty$ bound for later use. Our idea of adding the above viscosity is based on the following reasons. As is known, linear transport equation preserves the regularity of initial data and inviscid Burgers' equation may formulate shock wave. Thus we add parabolic viscosity term to Burgers' equation only. Moreover, to gain the uniform bound of $u^\eps$, we not only make full use of Vlasov equation and flux term in Burgers equation but also carefully choose control function. The viscosity term $\eps(\int_\R f^\eps dv)_x$ is introduced due to derivatives of our control function.

For any $\sigma\in(0, 1),$ we use $C^{2+\sigma}(\R^2\times[0, T])$ and $C^{2+\sigma, 1+\frac{\sigma}{2}}(\R\times[0, T])$ to denote usual and parabolic H\"older space respectively. We now would like to consider the global existence and uniqueness of smooth solutions $(u^\eps, f^\eps)$ to the Cauchy problem \eqref{e:burgers-vlasov-viscosity}-\eqref{e:initial-viscosity} with $u^\eps(x, t)\in C^{2+\sigma, 1+\frac{\sigma}{2}}(\R\times[0, T])$, $0\leq f^\eps(x, v, t)\in C^{2+\sigma}(\R^2\times[0, T])$. For simplicity, in this section, the up index $\eps$ in $u^\eps$ and $f^\eps$ will be dropped.

\subsection{Local existence.} We first consider the local existence of smooth solution of the Cauchy problem \eqref{e:burgers-vlasov-viscosity}-\eqref{e:initial-viscosity}. Let
\begin{align*}
G(x, t)=\begin{cases}
\delta(x), & t=0,\\
\frac{1}{\sqrt{4\pi\eps t}}e^{\frac{-x^2}{4\eps t}}, & t>0,
\end{cases}
\end{align*}
denote the kernel of the homogeneous heat equation $u_t=\eps u_{xx}.$  Then from the Burgers' equation, i.e. the first equation in \eqref{e:burgers-vlasov-viscosity}, Duhamel principle tells us
\begin{align*}
u(x, t)=&\int\!\! G(x-y, t)u_0(y)dy+\int_0^t\!\!\int G(x-y, t-s)\left(\!\int fvdv-u\!\int fdv \right)(y, s)dyds\\
&+\int_0^t\!\!\int G(x-y, t-s)\left(\eps\int fdv-\frac{1}{2}u^2\right)_y(y, s)dyds.
\end{align*}
Integrating by parts gives
\begin{align}
u(x, t)=&\int\!\! G(x-y, t)u_0(y)dy+\int_0^t\!\!\int G(x-y, t-s)\left(\!\int fvdv-u\!\int fdv \right)(y, s)dyds \notag\\
&+\int_0^t\!\!\int G_y(x-y, t-s)\left(\eps\int fdv-\frac{1}{2}u^2\right)(y, s)dyds, \label{e:u-implicit}
\end{align}
which implicitly gives the solution to the Burgers' equation. For the Vlasov equation, i.e. the second equation in \eqref{e:burgers-vlasov-viscosity}, we rewrite it as
\begin{align*}
f_t+vf_x+(u-v)f_v=f,
\end{align*}
which is a transport equation. One can integrate along the backward characteristic curves
\begin{align}
\frac{d}{ds}X(s; x, v, t)=V(s; x, v, t),\quad &X(t; x, v, t)=x, \label{e:x-equation}\\
\frac{d}{ds}V(s; x, v, t)=u(X(s; x, v, t), s)-V(s; x, v, t), \quad &
V(t; x, v, t)=v,\label{e:v-equation}
\end{align}
to get
\begin{equation}\label{e:f-implicit}
f(x, v, t)=f_0(X(0; x, v, t), V(0; x, v, t))e^t.
\end{equation}
Note that for smooth $u,$ the system of \eqref{e:x-equation} and \eqref{e:v-equation} has a unique smooth solution $(X(s; x, v, t), V(s; x, v, t)):$
\begin{align}
&X(s; x, v, t)=x+\int_t^s V(\tau; x, v, t)d\tau, \label{e:x-solution}\\
&V(s; x, v, t)=ve^{t-s}+\int_t^s e^{\tau-s}u(X(\tau; x, v, t), \tau)d\tau. \label{e:v-solution}
\end{align}
Hence, from \eqref{e:u-implicit} and \eqref{e:f-implicit} we construct the approximate solutions of \eqref{e:burgers-vlasov-viscosity} in the following way. Set $u^{(0)}=u_0^\eps, f^{(0)}=f_0^\eps,$ then there exists a {$K=C(\eps, M_0)$ such that for any $t>0$}
\begin{align*}
\|u^{(0)}\|_{C^{2+\sigma, 1+\frac{\sigma}{2}}(\R\times[0, t])}\leq K, \quad \|f^{(0)}\|_{C^{2+\sigma}(\R^2\times[0, t])}\leq K, \quad |\text{supp} f^{(0)}|\leq K,
\end{align*}
{where we used the definition
\begin{align*}
 |\text{supp} f|=\max \{|x|,|v|: f(x,v,\cdot)>0\}.
 \end{align*}}
For $k\geq1,$ we define
\begin{equation*}
\begin{array}{ll}
u^{(k)}(x, t)&\disp=\int\!\! G(x-y, t){u_0}(y)dy\\
&\disp+\int_0^t\!\!\int G(x-y, t-s)\left(\!\int f^{(k-1)}vdv-u^{(k-1)}\!\int f^{(k-1)}dv \right)(y, s)dyds\\
&\disp+\int_0^t\!\!\int G_y(x-y, t-s)\left(\eps\int f^{(k-1)}dv-\frac{1}{2}(u^{(k-1)})^2\right)(y, s)dyds,
\end{array}
\end{equation*}
and
\begin{equation*}
\begin{array}{ll}
&\disp f^{(k)}(x, v, t)= f_0(X^{(k)}(0; x, v, t), V^{(k)}(0; x, v, t))e^t,
\end{array}
\end{equation*}
where $(X^{(k)}(s; x, v, t), V^{(k)}(s; x, v, t))$ is defined using  \eqref{e:x-equation} and \eqref{e:v-equation} as
\begin{align*}
\frac{d}{ds}X^{(k)}(s; x, v, t)&=V^{(k)}(s; x, v, t),\\
 X^{(k)}(t; x, v, t)&=x, \\
\frac{d}{ds}V^{(k)}(s; x, v, t)&=u^{(k-1)}(X^{(k)}(s; x, v, t), s)-V^{(k)}(s; x, v, t), \\
V^{(k)}(t; x, v, t)&=v,
\end{align*}
{and }$(x, v)\in \text{supp} f^{(k)}.$
It is easy to see that $(X^{(k)}, V^{(k)})$ is well-defined. Thus $f^{(k)}$ and $u^{(k)}$ {make sense}. Besides, one can also see that $(u^{(k)}, f^{(k)})$ solves the following approximate equations
\begin{equation}\label{e:k-equation}
\left\{
\begin{array}{ll}
\displaystyle u^{(k)}_t+u^{(k-1)} u^{(k-1)}_x=\eps u^{(k)}_{xx}+\eps\left(\!\int f^{(k-1)} dv\right)_x+\int f^{(k-1)} vdv-u^{(k-1)}\int f^{(k-1)} dv, \\
\displaystyle f^{(k)}_t+v f^{(k)}_x+(u^{(k-1)}-v)f^{(k)}_v=f^{(k)},
\end{array}
\right.
\end{equation}

Obviously, from \eqref{e:initial-viscosity}, for any $(x, v)\in \text{supp} f^{(1)}$,
$$|X^{(1)}(0; x, v, t)|+|V^{(1)}(0; x, v, t)|\leq \eps^{-1/6}.$$
{Using \eqref{e:x-solution}, \eqref{e:v-solution} and the bound of $u^{(0)}$, we have
\begin{align*}
&|v|\leq |V^{(1)}(0; x, v, t)e^{-t}|+\left|\int_0^t e^{\tau-t}u^{(0)}(X^{(1)}(\tau; x, v, t), \tau)d\tau\right|\leq \varepsilon^{-{1/ 6}}+K\leq 2K, \\
&|x|\leq |X^{(1)}(0; x, v, t)|+\left|\int_0^tV^{(1)}(\tau; x, v, t)d\tau\right|\leq \eps^{-1/6}+3Kte^t\leq 2K,
\end{align*}
provided $t$ is sufficient small, where we have used the fact that
\begin{align*}
|V^{(1)}(s; x,v,t)|\leq & |v e^{t-s}|+\left|\int_t^s e^{\tau-s}u^{(0)}(X^{(1)}(\tau; x,v,t),\tau)d\tau\right| \\
\leq & e^t(2K+\|u^{(0)}\|_{L^\infty})\leq 3Ke^t, ~~for ~ 0\leq s\leq t.
\end{align*}
Thus  $|\text{supp}f^{(1)}|\leq 2K.$
%Besides, again from \eqref{e:x-solution} and \eqref{e:v-solution}, we can take $t$ smaller such that for any $(x, v)\in \text{supp}f^{(1)}$
%{\color{red}$$\|X^{(1)}\|_{C^{2+\sigma}([0, t])}\leq 6K, \quad \|V^{(1)}|_{C^{2+\sigma}([0, t])}\leq 6K.$$}
Moreover, we have the following conclusion.
\begin{lemma}\label{l:constraction}
There exists a small $t_0>0$ such that the sequence $\{(u^{(k)}, f^{(k)})\}_{k\geq0}$ constructed above is a  contraction sequence in
\begin{align*}
\mathcal{S}(t)=\{(u, f)| \|u\|_{C^{2+\sigma, 1+\frac{\sigma}{2}}(\R\times[0, t])}\leq 2K, \quad
\|f\|_{C^{2+\sigma}(\R^2\times[0, t])}\leq 2K, \quad |\text{supp}f|\leq 2K\}
\end{align*}
%and associated characteristic curves $\{(X^{(k)}, V^{(k)})\}_{k\geq1}$ for $f^{(k)}$ are in
%$$\mathcal{C}(t)=\{(X, V)|\|X\|_{C^{2+\sigma}([0, t])}\leq 6K, \quad \|V\|_{C^{2+\sigma}([0, t])}\leq 6K\}$$
for {all} $t\in(0, t_0)$.
\end{lemma}
\begin{proof}
It is easy to see that $(u^{(0)}, f^{(0)})\in S(t)$. Suppose that for $k\geq1$, $(u^{(k-1)}, f^{(k-1)})$ has been shown in $\mathcal{S}(t)$. We then estimate $(u^{(k)}, f^{(k)})$.

Taking derivatives with respect to $x$, one has for $\ell=0, 1, 2,$
\begin{align*}
\dx^\ell u^{(k)}(x, t)=&\int\!\! \dx^\ell G(x-y, t){u_0}(y)dy\\
&+\int_0^t\!\!\int \dx^\ell G(x-y, t-s)\left(\!\int f^{(k-1)}vdv-u^{(k-1)}\!\int f^{(k-1)}dv \right)(y, s)dyds\\
&+\int_0^t\!\!\int \dx^\ell G_y(x-y, t-s)\left(\eps\int f^{(k-1)}dv-\frac{1}{2}(u^{(k-1)})^2\right)(y, s)dyds.
\end{align*}
Using the symmetry of $G(x-y),$ one further has
\begin{align*}
\dx^\ell u^{(k)}(x, t)=&\int\!\! \dx^\ell G(x-y, t){u_0}(y)dy\\
&+\int_0^t\!\!\int (-1)^\ell\dy^\ell G(x-y, t-s)\left(\!\int f^{(k-1)}vdv-u^{(k-1)}\!\int f^{(k-1)}dv \right)(y, s)dyds\\
&+\int_0^t\!\!\int (-1)^\ell\dy^\ell G_y(x-y, t-s)\left(\eps\int f^{(k-1)}dv-\frac{1}{2}(u^{(k-1)})^2\right)(y, s)dyds.
\end{align*}
Integration by parts gives
\begin{align*}
\dx^\ell u^{(k)}(x, t)=&\int\!\! G(x-y, t)\dy^\ell {u_0(y)dy}\\
&+\int_0^t\!\!\int G(x-y, t-s)\left(\dy^\ell\!\int f^{(k-1)}vdv\right)(y, s)dyds\\
&-\int_0^t\!\!\int G(x-y, t-s)\dy^\ell\left( u^{(k-1)}\!\int f^{(k-1)}dv\right)(y, s)dyds\\
&+\int_0^t\!\!\int G_y(x-y, t-s)\dy^\ell\left(\eps\int f^{(k-1)}dv\right)(y, s)dyds,\\
&-\int_0^t\!\!\int G_y(x-y, t-s)\dy^\ell\left(\frac{1}{2}(u^{(k-1)})^2\right)(y, s)dyds.
\end{align*}
 For any $(x, v)\in \text{supp} f^{(k)}$, $t<t_1\ll1$, one has
$$|X^{(k)}(0; x, v, t)|+|V^{(k)}(0; x, v, t)|\leq \eps^{-1/6},$$
and
\begin{align*}
&|v|\leq |V^{(k)}(0; x, v, t)e^{-t}|+\left|\int_0^t e^{\tau-t}u^{(k-1)}(X^{(k)}(\tau; x, v, t), \tau)d\tau\right|\leq K+ 2Kt\leq 2K, \\
&|V^{(k)}(s; x, v, t)|\leq |v e^{t-s}|+|\int_t^s e^{\tau-s}u^{(k-1)}(X^{(k)}(\tau; x,v,t),\tau)d\tau| \leq  4e^t K\leq6K, \\
&|x|\leq |X^{(k)}(0; x, v, t)|+\left|\int_0^tV^{(k)}(\tau; x, v, t)d\tau\right|\leq K+{4Kte^t\leq 2K}\\
&|X^{(k)}(s; x, v, t)|\leq |x|+t|V^{(k)(s; x, v, t)}|\leq 2K+4te^tK\leq6K
\end{align*}
Thus
\begin{equation}\label{e:support-fk}
|\text{supp}f^{(k)}|\leq 2K.
\end{equation}
% and we can apply similar method to gain that for any $(x, v)\in \text{supp}f^{(k)}$
%\begin{equation}\label{e:xvk-estimate}
%\|X^{(k)}\|_{C^{2+\sigma}([0, t])}\leq 6K, \quad \|V^{(k)}\|_{C^{2+\sigma}([0, t])}\leq 6K.
%\end{equation}}
 Noticing the estimates of heat kernel
\begin{align*}
\int G(x, t)dx=1, \quad \int |G_x(x, t)| dx\leq \frac{C}{\sqrt{\eps t}},
\end{align*}
using the bound of $\text{supp} f^{(k-1)}$, one has
$$\|u^{(k)}(\cdot, t)\|_{C^2(\R)}\leq K+C(\eps, K)(t+\sqrt{t}),$$
then further gets
\begin{equation}\label{e:uk-estimate-1}
\|u^{(k)}\|_{C^{2+\sigma, 1+\frac{\sigma}{2}}(\R\times[0, t])}\leq K+C(\eps, K)(t+\sqrt{t})\leq 2K,
\end{equation}
provided $t<t_2\ll1$.

For the estimates of $f^{(k)} $ and for any $0\leq|\alpha|\leq 2,$ taking derivatives one has
\begin{equation}\label{e: derivatives-f^k}
(\partial^\alpha f^{(k)})_t+v(\partial^\alpha f^{(k)})_x+(u^{(k-1)}-v)(\partial^\alpha f^{(k)})_v=B(\alpha)\partial^\alpha f^{(k)}+ B(\alpha'),
\end{equation}
where $\partial^\alpha$ denotes the mix derivatives of $x, v$ and $t$ with order $|\alpha|$, $B(\alpha)$ is a linear function of $v$ and $\partial_{x,t}^\beta u^{(k-1)},$ with $|\beta|\leq|\alpha|$ and  $B(\alpha')$ is a  linear combination of $\partial^{\alpha'}f^{(k)}$ with $|\alpha'|<|\alpha|,$ whose coefficients are linear functions of $v$ and $\partial_{x,t}^\beta u^{(k-1)},$ with $|\beta|\leq|\alpha|.$
Integrating (\ref{e: derivatives-f^k}) along the characteristic curves $(X^{(k)}, V^{(k)})$, one has
\begin{align*}
\partial^\alpha f^{k}=& (\partial^\alpha f_0)(X^{(k)}, V^{(k)})e^{\int_0^tB(\alpha)(X^{(k)}, V^{(k)})d\tau}\\
&+\int_0^t e^{\int_\tau^tB(\alpha)(X^{(k)}, V^{(k)})ds}B( \alpha')(X^{(k)}, V^{(k)})d\tau,
\end{align*}
which then implies
$$\| f^{(k)}\|_{C^2(\R^2\times[0, t])}\leq Ke^{tC(\eps, K)}+C(\eps, K)te^{C(\eps, K)t}\leq 2K.$$
We further gain the following H\"older estimate
\begin{equation}\label{e:fk-estimate-1}
\|f^{(k)}\|_{C^{2+\sigma}(\R^2\times[0, t])}\leq 2K,
\end{equation}
provided $t<t_3\ll1$.
Hence from \eqref{e:support-fk}, \eqref{e:uk-estimate-1} and \eqref{e:fk-estimate-1}, one has
$$(u^{(k)}, f^{(k)})\in \mathcal{S}(t).$$
Moreover, consider the equation for $u^{(k)}-u^{(k-1)}$ and $f^{(k)}-f^{(k-1)}$, similar to above calculation, we get
\begin{align*}
\|u^{(k+1)}-u^{(k)}\|_{C^{2+\sigma, 1+\frac{\sigma}{2}}(\R\times[0, t])}\leq& C(\eps, K)t\|u^{(k)}-u^{(k-1)}\|_{C^{2+\sigma, 1+\frac{\sigma}{2}}(\R\times[0, t])}\\
\leq&\frac{1}{2}\|u^{(k)}-u^{(k-1)}\|_{C^{2+\sigma, 1+\frac{\sigma}{2}}(\R\times[0, t])}\\
\|f^{(k+1)}-f^{(k)}\|_{C^{2+\sigma}(\R^2\times[0, t])}\leq& C(\eps, K)(e^{t}-1)\|f^{(k)}-f^{(k-1)}\|_{C^{2+\sigma}(\R^2\times[0, t])}\\
\leq&\frac{1}{2}\|f^{(k)}-f^{(k-1)}\|_{C^{2+\sigma}(\R^2\times[0, t])},
\end{align*}
provided $t<t_4\ll1$. Let $t_0=\min\{t_i, i=1, \cdots, 4\}.$ Then we gain $(u^{(k)}, f^{(k)})$ is a contraction sequence in $\mathcal{S}(t)$ with $t\in(0, t_0)$ and end the proof.
\end{proof}

Applying fixed point theorem to $(u^{(k)}, f^{(k)}), $ combing with Lemma \ref{l:constraction}, one gains that there exists a pair of functions $(u, f)$ such that
\begin{align*}
&u^{(k)}\rightarrow u \text{ in }C^{2+\sigma, 1+\frac{\sigma}{2}}(\R\times[0, t_0]), \\
&f^{(k)}\rightarrow f \text{ in } C^{2+\sigma}(\R^2\times[0, t_0]),
\end{align*}
 and $(u, f)$ is the unique smooth solution of Cauchy problem \eqref{e:burgers-vlasov-viscosity}-\eqref{e:initial-viscosity} by taking limit in \eqref{e:k-equation}.

\subsection{Uniform estimates.} We will apply maximum principles of parabolic equation and  transport equation to bound $\|u\|_{L^\infty(\R\times[0, T])}$ and $\|f\|_{L^\infty([0, T], (1+v^2)L^1(\R^2))}.$

\begin{lemma}\label{l:uniform}
For the approximate solutions constructed above, there exists a constant $M$ depending only on $M_0$ such that
\eqref{e:solution-bound} holds for $(u^\eps, f^\eps).$
\end{lemma}

%\begin{remark}
%The uniform estimates established here  neither depend on $t$ nor on $\eps.$
%\end{remark}
\begin{proof}
Obviously, from \eqref{e:initial-viscosity} and \eqref{e:f-implicit} one has
$f(x, v, t)\geq0.$ For any $(x, v)\in\supp{f},$ by  \eqref{e:initial-viscosity} and \eqref{e:v-solution}, one has
\begin{align}
|v|\leq&|V(0; x, v, t)e^{-t}|+\left|\int_0^t e^{\tau-t}u(X(\tau; x, v, t), \tau)d\tau\right|\notag\\
\leq&\eps^{-{1}/{6}}+\hat{C}(\eps, M_0, T), \label{e:v-bound}
\end{align}
after priori assuming
$$\|u\|_{L^{\infty}}\leq \hat{C}(\eps, M_0, T) \text{ for some large enough } \hat{C}(\eps, M_0, T).$$
Besides,
\begin{align}
|x|\leq &|X(0; x, v, t)|+\left|\int_0^tV(\tau; x, v, t)d\tau\right|\notag\\
\leq&\eps^{-1/6}+(2\hat{C}(\eps, M_0, T)+\eps^{-1/6})Te^T. \label{e:x-bound}
\end{align}
Hence by \eqref{e:v-bound}-\eqref{e:x-bound}, $f$ enjoys compact support (depending on $\eps, T$) and
\begin{align*}
\lim_{|x|\rightarrow\infty \text{ or } |v|\rightarrow\infty}f(x, v, t)=0.
\end{align*}
Thus integrating the Vlasov equation over $\R^2\times[0, t]$ with respect to $(x, v, t)$ gives
\begin{equation}
\label{e:f-bound-L1}
\int\!\!\int f(x, v, t)dxdv=\int\!\!\int f_0(x, v)dxdv\leq M_0,
\end{equation}
where we have used \eqref{e:initial-condition}. Moreover, integration over $\R$ with respect to $v$ also gives
\begin{equation}\label{e:int-f-equation}
\left(\int fdv\right)_t+\left(\int vfdv\right)_x=0.
\end{equation}
Define control function
\begin{align*}
\psi(x, t)=\int_x^\infty\!\!\int f(y, v, t) dvdy,
\end{align*}
then using \eqref{e:int-f-equation} one has
\begin{align*}
\psi_t&=\int_x^\infty\!\!\left(\int f(y, v, t) dv\right)_tdy=-\int_x^\infty\!\!\left(\int vf(y, v, t) dv\right)_ydy=\!\int fvdv, \\
\psi_x&=-\int f dv, \quad \psi_{xx}=-\left(\int fdv\right)_x.
\end{align*}
Thus one is also able to derive the equation for $u-\psi$
\begin{align*}
(u-\psi)_t+u(u-\psi)_x=&\eps u_{xx}+\eps\left(\!\int fdv\right)_x+\int fvdv-u\int f dv\\
&-\psi_t-u\psi_x\\
=&\eps(u-\psi)_{xx}.
\end{align*}
Applying maximum principle for the above parabolic equation with respect to $u-\psi$, we obtain
\begin{align*}
\|u-\psi\|_{L^\infty}\leq \|u_0-\psi(x, 0)\|_{L^\infty}\leq M_0+\int\!\!\int f_0dvdx\leq 2M_0.
\end{align*}
Thus, we have
\begin{equation}
\label{e:u-bound}
\|u(x, t)\|_{L^\infty}\leq\|\psi\|_{L^\infty}+\|u-\psi\|_{L^\infty} \leq 3M_0
\end{equation}
by using \eqref{e:f-bound-L1}, which also close our priori assumption
$$\|u\|_{L^\infty}\leq 3M_0< \hat{C}(\eps, M_0, T).$$
 Besides, similar to the calculations of \eqref{e:v-bound} and \eqref{e:x-bound}, one has
\begin{equation}\label{e:xv-bound}
|v|\leq\eps^{-1/6}+3M_0, \quad {|x|\leq\eps^{-1/6}+(\eps^{-1/6}+2M_0)Te^T.}
\end{equation}

Furthermore, multiplying the Vlasov equation by $v^2$  and integrating over $\R^2$ with respect to $(v, x)$ we have
\begin{align*}
\frac{d}{dt}\int\!\!\int f v^2dvdx&=\int\!\!\int (2f vu-2fv^2)dvdx\leq \int\!\!\int f (u^2-v^2)dvdx\\
&\leq \|u\|^2_{L^\infty}\int\!\!\int f dvdx-\int\!\!\int f v^2dvdx\\
&\leq 9M_0^3-\int\!\!\int f v^2dvdx,
\end{align*}
where we have used \eqref{e:f-bound-L1}. Gronwall inequality yields
\begin{equation}\label{e:f-bound-v2L1}
\int\!\!\int f v^2dvdx\leq e^{-t}\int\!\!\int f_0v^2dvdx+9M_0^3(1-e^{-t})\leq M_0+9M_0^3.
\end{equation}
thus \eqref{e:f-bound-L1}, \eqref{e:u-bound} and \eqref{e:f-bound-v2L1}  conclude the present Lemma.
\end{proof}

\subsection{Conclusion.} Standard theory of quasilinear parabolic equation (see \cite{LSU}) can be applied to the equation for $u-\psi$
 $$(u-\psi)_t+u(u-\psi)_x=\eps(u-\psi)_{xx}$$
to get
$$\|(u-\psi)_x\|_{C^0(\R)}\leq C(\eps, T),$$
where we have used the uniform bound on $\|u\|_{C^0(\R)}\leq C$ and $C(\eps, T)$ is increasing function of $T.$ So we have
$$\|u_x\|_{C^0(\R)}\leq C(\eps, T).$$
Then using the Vlasov equation, we can also get $$\|f_x\|_{C^0(\R^2)}+\|f_v\|_{C^0(\R^2)}\leq C(\eps, T).$$
By standard bootstrap argument we have the following estimate
$$\|u(x, t)\|_{C^{2+\sigma, 1+\frac{\sigma}{2}}(\R\times[0, T])}+\|f(x, v, t)\|_{C^{2+\sigma}(\R^2\times[0, T])}\leq C(\eps, T),$$
With the local existence in Subsection 2.1, the global existence of solution $u(x, t)\in C^{2+\sigma, 1+\frac{\sigma}{2}}(\R\times[0, T])$ and $0\leq f(x, v, t)\in C^{2+\sigma}(\R^2\times[0, T])$ to the Cauchy problem \eqref{e:burgers-vlasov-viscosity}-\eqref{e:initial-viscosity} is obtained. Thus we got the following conclusion.
\begin{theorem}%\label{t:approximate-existence}
For any $T>0,$ any fixed $\eps,$ there exist a unique global solution $u^\eps(x, t)\in C^{2+\sigma, 1+\frac{\sigma}{2}}(\R\times[0, T])$, $0\leq f^\eps(x, v, t)\in C^{2+\sigma}(\R^2\times[0, T])$ to Cauchy problem \eqref{e:burgers-vlasov-viscosity}-\eqref{e:initial-viscosity}.
\end{theorem}

%%%%%%%%%%%%%%%%%%%%%%%%%
%%%%%%%%%%%%%%%%%%%%%%%%
\section{Proof of main theorem }\label{s:proof}
%%%%%%%%%%%%%%%%%%%%%
%%%%%%%%%%%%%%%%%%%

In this section, we will establish the convergence of $(u^\eps, f^\eps),$ whose limit is just an $L^\infty-L^1$ weak solution to Cauchy problem \eqref{e:burgers-vlasov}-\eqref{e:initial}.

%%%%%%%%%
\subsection{Limit of functions.}
We first show some convergence results related to $f^\eps.$ Recall the well-known weak $L^1$ compactness framework, i.e. Dunford-Pettis  Theorem (see \cite{LMR} Theorem 8 or \cite{B}, page 167 ).
\begin{proposition}{\bf [Dunford-Pettis]}\label{p:compactness-1}
A sequence $\{f^\eps\}$ is weakly compact in $L^1(\R^2)$ if and only if  $\{f^\eps\}$  satisfies the following conditions:
\begin{itemize}
\item[(1)] The sequence $f^\eps$ is equibounded in $L^1(\R^2)$, i.e.
\begin{align*}
\sup_{\eps}\|f^\eps\|_{L^1(\R^2)}<\infty.
\end{align*}
\item[(2)] The sequence $f^\eps$ is equiintegrable, i.e.
\begin{itemize}
\item[(2a)] For any $\delta>0,$ there exists measurable set $A\subset\R^2$ with $|A|<\infty $ such that
\begin{align*}
\int_{\R^2\setminus A}f^\eps dxdv<\delta.
\end{align*}
\item[(2b)]  For any $\delta>0,$ there exists $\kappa>0$ such that for any measurable set $E\subset\R^2,$ with $|E|\leq\kappa,$ there holds
\begin{align*}
\int_Ef^\eps dxdv<\delta.
\end{align*}
\end{itemize}
\end{itemize}
\end{proposition}

We shall verify (1) and (2) for $f^\eps(x, v, t)$ with any fixed $t>0.$

\textit{ Verification of (1):} Obviously, from Lemma \ref{l:uniform},
one finds that
\begin{equation}\label{e:f-bound-weight}
\|f^\eps\|_{L^\infty([0, T], (1+v^2)L^1(\R^2))}\leq M,
\end{equation}
which means $f^\eps(\cdot, \cdot, t)$ is uniformly equibounded with respect to $\eps$ and $t$ in $L^1(\R^2).$

\textit{ Verification of (2a):} For any $\delta>0,$ we can choose $A=\{(x, v)|  |x|\leq \Lambda, |v|\leq \Lambda\}$ with $\Lambda\geq\frac{M}{\delta}$ where $M$ comes from \eqref{e:f-bound-weight}, so we have for any $t>0,$
\begin{align*}
\int_{\R^2\setminus A}f^\eps(x, v, t)dxdv\leq\frac{1}{\Lambda^2+1}\int_{\R^2\setminus A} f^\eps(x, v, t)(1+v^2) dxdv\leq\delta,
\end{align*}
which implies that  (2a) is satisfied by $f^\eps(x, v, t).$

\textit{Verification of (2b):} By the fact that  $f^\eps_0\rightharpoonup f_0$ in $L^1(\R^2),$  Dunford-Pettis theorem tells us that for any $\delta,$ there exists $\kappa_0$ such that for any $E_0\subset\R^2$ with $|E_0|\leq \kappa_0$  it holds that
\begin{align*}
\int_{E_0}f^\eps_0(x, v)dxdv\leq\delta.
\end{align*}
%for Vlasov equation
%$$f^\eps_{t}+(f^\eps v)_{x}+(f^\eps (u^\eps-v))_{v}=0$$
%with $u^\eps$ being given smooth functions,  then the equation can be rewritten as
%\begin{equation*}
%f^\eps_{t}+vf^\eps_{x}+(u^\eps-v)f^\eps_{v}=f^\eps,  \quad  f^\eps(x, v, 0)=f^\eps_{0}(x,v),
%\end{equation*}
%which has a unique solution
%\begin{equation}\label{e:f}
%\displaystyle f^\eps(x,v,t)=f^\eps_{0}(X^\eps(0; x,v,t), V^\eps(0; x,v,t))e^t,
%\end{equation}
%where the backward characteristic curves $X(s;x,v,t), V(t;x,v,t)$ are defined by
%\begin{align}
%&\frac{dX^\eps(s;x,v,t)}{ds}=V^\eps(s;x,v,t), \quad  X^\eps(t;x,v,t)=x \label{e:x} \\
%&\frac{dV(s;x,v,t)}{ds}=u^\eps(s,X^\eps(s;x,v,t))-V^\eps(s;x,v,t), \quad  V^\eps(t;x,v,t)=v. \label{e:v}
%\end{align}
%It is not hard to see that $X^\eps, V^\eps$ are well defined from the theory of ordinary differential equation.
On the other hand, considering the variable transformation
$$\mathcal{J}:(x, v)\mapsto(X^\eps, V^\eps),$$
from \eqref{e:x-equation} and \eqref{e:v-equation}, one is able to show that the Jacobian
$$J(\tau)=\det \nabla_{x, v}(X^\eps, V^\eps)$$
of map $\mathcal{J}$ is positive and satisfies the following ODE
\begin{equation*}
\left\{
\begin{array}{ll}
\displaystyle \frac{d J(\tau)}{d\tau}=-J(\tau),\\
\displaystyle J(t)=1.
\end{array}
\right.
\end{equation*}
Then we have $J(\tau)= e^{t-\tau}.$ For any $T\in(0, \infty)$, one can take $\kappa=e^{-T}\kappa_0.$ Then for any measurable set $E\subset\R^2$ with $|E|\leq \kappa,$ one has $|\mathcal{J}(0)(E)|\leq e^{t}e^{-T}\kappa_0\leq\kappa_0$ for any $t\in[0, T].$ Hence one further has
\begin{align*}
\int_{\mathcal{J}(0)(E)}f^\eps_0(x, v)dxdv\leq\delta.
\end{align*}
Finally, with \eqref{e:f-implicit}, one gains
\begin{align*}
\int_Ef^\eps(x, v, t)dxdv&=\int_{\mathcal{J}(0)(E)}f^\eps_{0}(X^\eps(0; x,v,t), V^\eps(0; x,v,t))e^t J(0)^{-1}dX^\eps dV^\eps\\
&=\int_{\mathcal{J}(0)(E)}f^\eps_{0}(X^\eps(0; x,v,t), V^\eps(0; x,v,t))dX^\eps dV^\eps\leq\delta,
%&=\int_{E_0}f^\eps_{0}(X^\eps(0; x,v,t), V^\eps(0; x,v,t))dX^\eps dV^\eps\\
\end{align*}
which then gives (2b).

Therefore, applying Proposition \ref{p:compactness-1}, we get some subsequence (for simplicity we still denote) $f^\eps$ and a nonnegative function $f\in L^{\infty}([0, T], L^1(\R^2))$ such that
\begin{equation}\label{e:f-weak-limit}
f^{\eps}(x, v, t)\rightharpoonup f(x, v, t)\text{ weakly in } L^1(\R^2), \text{ for any } t>0
\end{equation}
and
\begin{equation}\label{e:f-limit-1}
\int\!\!\int\!\! fdxdv\leq M.
\end{equation}

Next we shall show the convergence of $\int vf^\eps dv$. In fact, using the convexity of kinetic energy and weak convergence of $f^\eps,$ one is able to get from \eqref{e:f-bound-v2L1} that
\begin{equation}\label{e:f-limit-2}
\int\!\!\int f v^2dvdx\leq \liminf_{\eps\rightarrow0}\int\!\!\int f^\eps v^2dvdx\leq  M.
\end{equation}

For the weak convergence of $\int f^\eps vdv,$ we utilize the approach in \cite{ACF}. Denote $\mathbf{1}_{[-1, 1]}(s)$ as $\omega(s),$ for any $\va\in C_c^\infty(\R\times[0, T]),$ for arbitrary $L>0,$ one can derive
\begin{align*}
&\int_0^T\!\!\!\int\left(\int vf^\eps dv-\int vfdv\right)\va dxdt\\
=&\int_0^T\!\!\int\left(\int \omega(\frac{v}{L})vf^\eps dv-\int \omega(\frac{v}{L})vfdv\right)\va dxdt\\
&+\int_0^T\!\!\int\!\!\int vf^\eps (1-\omega(\frac{v}{L}))dv\va dxdt\\
&-\int_0^T\!\!\int \!\!\int vf(1-\omega(\frac{v}{L}))dv\va dxdt.
\end{align*}
By (\ref{e:f-weak-limit}), i.e. the weak convergence of $f^\eps$ in $L^1$ , one can get that  the first term on the right hand side converges to 0 as $\eps\rightarrow0.$  For the last two terms, using \eqref{e:f-bound-v2L1} and \eqref{e:f-limit-2} we also has
\begin{align*}
&\left|\int_0^T\!\!\int\!\!\int vf^\eps (1-\omega(\frac{v}{L}))dv\va dxdt\right|\leq\frac{\|\va\|_{L^\infty}}{L}\int_0^T\!\!\!\int\!\!\int f^\eps v^2 dvdxdt\leq {\frac{MT\|\va\|_{L^\infty}}{L}},\\
&\left|\int_0^T\!\!\int \!\!\int vf(1-\omega(\frac{v}{L}))dv\va dxdt\right|\leq\frac{\|\va\|_{L^\infty}}{L}\int_0^T\!\!\!\int\!\!\int f v^2 dvdxdt\leq {\frac{MT\|\va\|_{L^\infty}}{L}},
\end{align*}
both of which go to 0 when $L\rightarrow\infty.$ Therefore, we have
\begin{equation}\label{e:fv-weak-limit}
\int f^\eps vdv\rightarrow \int vfdv \text{ in the sense of distribution.}
\end{equation}
Similarly, we also have
\begin{equation}\label{e:int-f-limit}
\int f^\eps dv\rightarrow \int fdv \text{ in the sense of distribution.}
\end{equation}

Now, we consider the convergence of $u^\eps$. We also recall the $L^\infty$ compensated compactness framework and Murat's Lemma:
\begin{proposition}\label{p:compactness}(\cite{Chen00})
Assume that a sequence $u^\eps(x, t)$ satisfies
\begin{align*}
\|u^\eps\|_{L^\infty}\leq C,
\end{align*}
and
\begin{align*}
\eta(u^\eps)_t+q(u^\eps)_x \text{ is compact in }H^{-1}_{loc}(\mathbb{R}\times[0, T])
\end{align*}
for any entropy pair $(\eta, q)$ with $\eta'(u)u=q'(u)$ (or two special entropy pairs in Theorem 2.7 of \cite{Chen00}). Then there exist
a subsequence $\{u^{\eps_k}\}_{k=1}^\infty\subset\{u^\eps\}_{\eps>0}$ and function $u(x, t)$ such that
\begin{align*}
u^{\eps_{k}}\rightarrow u, \quad (u^{\eps_{k}})^2\rightarrow u^2, \text{a.e. as } k\rightarrow\infty.
\end{align*}
\end{proposition}
\begin{lemma}\label{l:murat} (\cite{Chen86, Ta79})
Let $\Omega\in\mathbb{R}^n$ be a open bounded subset, then
\begin{align*}
(\text{compact set of } W^{-1,a}_{loc}(\Omega) )\cap
(\text{bounded set of } W_{loc}^{-1,b}(\Omega))
\subset (\text{compact set of }H^{-1}_{loc}(\Omega)),
\end{align*}
where $a$ and $b$ are constants satisfying $1<a\leq2<b.$
\end{lemma}

From uniform estimate \eqref{e:u-bound}, there exists a $u(x, t)\in L^\infty(\R\times[0, T])$ such that
\begin{equation}\label{e:u-limit}
\|u\|_{L^\infty(\R\times[0, T])}\leq M \text{ and } u^\eps(x, t)\rightharpoonup u(x, t), \text{ weak * in } L^\infty(\R\times[0, T]).
\end{equation}
To prove the strong convergence of $u^\eps$, we need to get entropy dissipation estimate. For any compact set $\Omega\subset\R\times[0, T],$ for any $\va\in C^{\infty}_c(\R\times[0, T])$ with $\va|_\Omega=1,$ multiplying the first equation in \eqref{e:burgers-vlasov-viscosity} by $u^\eps\va^2$,  integrating over $\R\times[0, T]$, one has
\begin{align*}
\int_0^T\!\!\int\left[\frac{1}{2}(u^\eps)^2_t\va^2+\frac{1}{3}(u^\eps)^3_x\va^2 \right]dxdt=&\int_0^T\!\!\int\left[\eps u^\eps_{xx}u^\eps\va^2+\eps\left(\!\int f^\eps dv\right)_xu^\eps\va^2 \right]dxdt\\
&+\int_0^T\!\!\int\left[u^\eps\va^2\int f^\eps vdv-\va^2 (u^\eps)^2\int f^\eps dv \right]dxdt.
\end{align*}
Integrating by parts gives
\begin{align*}
\int_0^T\!\!\int\eps\va^2 (u_x^\eps)^2dxdt=
&\int_0^T\!\!\int\left[(u^\eps)^2\va\va_t+\frac{2}{3}(u^\eps)^3\va\va_x-2\eps \va\va_x u^\eps\!\!\int f^\eps dv\right]dxdt\\
&+\int_0^T\!\!\int\left[u^\eps\va^2\int f^\eps vdv-\va^2 (u^\eps)^2\int f^\eps dv \right]dxdt\\
&-\int_0^T\!\!\int\left[2\eps\va\va_xu^\eps u^\eps_x+\eps\va^2 u^\eps_x\!\int f^\eps dv\right] dxdt.
\end{align*}
From Lemma \ref{l:uniform}, one is easy to get the first two terms are bounded.  Applying H\"older inequality and \eqref{e:f-bound-weight}, the last term is bounded by
\begin{align*}
\frac{1}{2}\int_0^T\!\!\int_\Omega\eps (u^\eps_x)^2 dxdt+ M\eps\int_0^T\!\!\int\!\va^2\left(\int f^\eps dv\right)^2 dxdt+ C(M_0, T).
\end{align*}
From \eqref{e:initial-viscosity} and \eqref{e:f-implicit}, one has
$$\|f\|_{L^\infty}\leq e^t\|f_0^\eps\|_{L^\infty}\leq\eps^{-{1}/{6}}e^T.$$
By \eqref{e:xv-bound}, we further have
\begin{equation}\label{e:int-f-l2}
\eps\int_0^T\!\!\int\va^2\!\left(\int f^\eps dv\right)^2 dxdt\leq C(M_0, T)\eps^{1/3}.
\end{equation}
So we gain
\begin{equation}\label{e:u-derivative}
\int\!\!\int_\Omega\eps (u^\eps_x)^2 dxdt\leq C(M_0, T).
\end{equation}
Now we are ready to show the entropy dissipation of \eqref{e:burgers-vlasov-viscosity}.  For any weak entropy-entropy flux $(\eta, q)$ with $\eta\in C^2,$ one has
\begin{align*}
\eta(u^\eps)_t+q(u^\eps)_x=&\eps\eta(u^\eps)_{xx} -\eps\eta''(u^\eps)(u^\eps_x)^2+\eta'(u^\eps)\!\!\int f^\eps (v-u^\eps)dv\\
&\quad+\eps\eta'(u^\eps)\left(\int f^\eps dv\right)_x=:\sum_{k=1}^4I_k,
\end{align*}
which is got by multiplying \eqref{e:burgers-vlasov-viscosity} by $\eta'(u^\eps).$ Obviously, from the uniform boundedness of $u^\eps$ \eqref{e:u-bound}, using the estimates \eqref{e:f-bound-weight} and \eqref{e:u-derivative}, we obtain $I_2+I_3$ is bounded in $L^1_{loc}(\R\times[0, T])$. Thus by embedding theorem and Schauder theorem $I_2+I_3$ is compact in $W^{-1, \alpha}_{loc}(\R\times[0, T])$ with some $1<\alpha<2.$ For $I_1$, we also have
\begin{align*}
\left|\int_0^T\int\eps\eta(u^\eps)_{xx}\va dxdt\right|&=
\left|\!\!\int_0^T\!\!\int\eps\eta'(u^\eps)u^\eps_{x}\va_x dxdt\right|\\
&\leq M\sqrt{\eps}\left(\int_0^T\!\!\int_\Omega \eps (u^\eps_x)^2dxdt\right)^{1/2},
\end{align*}
for any  compact set $\Omega\subset\R\times[0, T],$ and any $\va\in C^{\infty}_c(\Omega).$ So we have $I_1$ is compact in $H^{-1}_{loc}(\R\times[0, T]).$ For $I_4,$ using \eqref{e:int-f-l2} and \eqref{e:u-derivative}, one also has
{\begin{align*}
&\left|\int_0^T\int\eps\eta'(u^\eps)\left(\int f^\eps dv\right)_x\va dxdt\right|\\
\leq&\left|\!\!\int_0^T\!\!\int\eps\eta'(u^\eps)\int f^\eps dv\va_x dxdt\right|
+\left|\!\!\int_0^T\!\!\int\eps\va\eta''(u^\eps)u^\eps_x\int f^\eps dvdxdt\right|\\
\leq&M\eps+M\left(\int_0^T\!\!\int\eps (u^\eps_x)^2dxdt\right)^{1/2}
\left(\eps\int_0^T\!\!\int\va^2 \left(\int f^\eps dv\right)^2dxdt\right)^{1/2}\\
\leq&M\eps^{1/6},
\end{align*}}
which also implies $I_4$ is compact in $H^{-1}_{loc}(\R\times[0, T]).$ Putting things together, one has
\begin{equation}
\label{e:eta-q-compact}
\eta(u^\eps)_t+q(u^\eps)_x \text{ is compact in } W^{-1, \alpha}_{loc}(\R\times[0, T]) \text{ for some } 1<\alpha<2.
\end{equation}
From the $L^\infty$ bound of $u^\eps,$  one also has
\begin{equation}
\label{e:eta-q-bound}
\eta(u^\eps)_t+q(u^\eps)_x \text{ is bounded in } W^{-1, \infty}_{loc}(\R\times[0, T]).
\end{equation}
Applying Murat's lemma (see Lemma \ref{l:murat}) to \eqref{e:eta-q-compact} and \eqref{e:eta-q-bound}, we finally have
\begin{equation}
\label{e:eta-q-h-compact}
\eta(u^\eps)_t+q(u^\eps)_x \text{ is compact in } H^{-1}_{loc}(\R\times[0, T]).
\end{equation}
Therefore, using $L^\infty $ compensated compactness framework (see Proposition \ref{p:compactness} or Theorem 2.7 in \cite{Chen00}), \eqref{e:u-limit} and \eqref{e:eta-q-h-compact}, we have
\begin{align}
&u^\eps(x, t)\rightarrow u(x, t),~~  \text{a.e. in } \R\times[0, T], \notag\\
&u^\eps(x, t)\rightarrow u(x, t), \text{ in } L^p_{loc}(\R\times[0, T]), \quad \forall p\in [1, \infty). \label{e:u-weak-limit}
\end{align}

%%%%%%%%%%%%%%%%%%%%%
\subsection{Limit of equations.}
%%%%%%%%%%%%%%%%%%%%%%%
To show $(u, f)$ is weak solution to \eqref{e:burgers-vlasov}, we need to verify \eqref{e:burgersweak} and \eqref{e:vlasovweak}. For simplicity, here we only show \eqref{e:burgersweak}, since \eqref{e:vlasovweak} can be treated similarly. Multiplying the first equation in \eqref{e:burgers-vlasov-viscosity} by $\phi\in C^\infty_c(\R\times[0, T))$, integrating over $\R\times[0, T]$ and using integration by parts, we obtain
\begin{align*}
&\int\phi(x, 0)u^\eps_0(x)dx+\int_0^T\!\!\!\int\left(u^\eps\phi_t+\frac{1}{2}(u^\eps)^2 \phi_x+\phi\int f^\eps(v-u^\eps)dv\right)dxdt\\
&\qquad +\eps\int_0^T\!\!\!\int \left(u^\eps_x\phi_x+\!\!\int f^\eps dv \phi_x \right)dxdt=0.
\end{align*}
For the last two term in the left hand side, using \eqref{e:u-derivative}, one can derive
\begin{align*}
\left|\eps\int_0^T\!\!\!\int u^\eps_x\phi_x dxdt\right|&\leq\sqrt{\eps}\left(\int_0^T\!\!\!\int\eps (u^\eps_x)^2dxdt\right)^{\frac{1}{2}}\left(\int_0^T\!\!\!\int\phi_x^2dxdt\right)^{\frac{1}{2}}\\
&\leq C(M_0, \|\phi\|_{H^1(\R)}, T)\sqrt{\eps},
\end{align*}
and
\begin{align*}
\left|\eps\int_0^T\!\!\!\int \phi_x\int fdv  dxdt\right|\leq C(M_0, T, \|\phi\|_{C^1(\R)})\eps,
\end{align*}
both of which go to 0 when $\eps\rightarrow0.$  It only remains to show the convergence of $ u^\eps\int f^\eps dv.$ In fact, observing that
\begin{align*}
&u^\eps\!\!\int f^\eps dv-u\!\!\int fdv\\
=(&u^\eps-u)\!\!\int f^\eps dv+\left(\int f^\eps dv-\int fdv\right)u,
\end{align*}
using \eqref{e:int-f-limit} and \eqref{e:u-weak-limit} we have
\begin{align*}
u^\eps\int f^\eps dv\rightarrow u\!\int fdv \text{ in the sense of distribution.}
\end{align*}
Thus we have \eqref{e:burgersweak}.

%To show $f$ is a weak solution to the Vlasov equation. Multiplying the second equation in \eqref{e:burgers-vlasov-viscosity} by $\psi$ with  test function $\psi\in C^\infty_c(\R^2\times(0, T))$ and integrating the result, we have
%\begin{align*}
%\int\!\!\int f^\eps_0(x, v)\psi(x, v, 0)dxdv+\int_0^T\!\!\int\!\!\int [f^\eps\psi_t+f^\eps v\psi_x+(u^\eps-v)f^\eps\psi_v]dvdxdt=0.
%\end{align*}
% Using \eqref{e:f-weak-limit} and \eqref{e:u-limit} one easily  gets \eqref{e:vlasovweak} for $(u, f)$ after taking $\eps\rightarrow0$. Note that \eqref{e:solution-bound} follows from \eqref{e:f-limit-1}, \eqref{e:f-limit-2} and \eqref{e:u-limit}.

%%%%%%%%%%%%%%%%
\subsection{Entropy inequality.}
%%%%%%%%%%%%%%
We shall also show entropy inequality for Burgers equation, i.e. \eqref{e:entropy-inequality}. Multiply the first equation in \eqref{e:burgers-vlasov-viscosity} by $\eta'(u^\eps)\va$, where $\eta$ is convex and  $\va\in C^\infty_c(\R\times(0, T))$ is nonnegative, and integrate the result over $\R\times(0, T)$ we get
\begin{align*}
&\int_0^T\!\!\!\int_\R\left(\eta(u^\eps)\va_t+q(u^\eps)\va_x+\va\eta'(u^\eps)\int_\R f^\eps(v-u^\eps)dv\right)dxdt\\
=&-\int_0^T\!\!\!\int\left(\eps u^\eps_{xx}+\eps\left(\int f^\eps dv\right)_x\right)\eta'(u^\eps)\va dxdt\\
=&\eps\int_0^T\!\!\!\int u^\eps_x \eta'(u^\eps)\va_x dxdt+\eps \int_0^T\!\!\!\int \eta''(u^\eps)(u^\eps_x)^2\va dxdt\\
&\quad +\int_0^T\!\!\!\int \eps\va_x\eta'(u^\eps)\int f^\eps dv dxdt+\int_0^T\!\!\!\int\eps\eta''(u^\eps)u^\eps_x\va\int f^\eps dvdxdt\\
%\geq& C(\va)\left[\left(\int\!\!\!\int_{\supp\va}\eps (u^\eps_x)^2dxdt\right)^{\frac{1}{2}}\left(\sqrt{\eps}+\left(\int\!\!\!\int_{\supp\va}\eps\left(\int fdv\right)^2dxdt\right)^{\frac{1}{2}}\right)\right]\\
\geq& -C(\va)(\eps+{\eps}^{1/2}+\eps^{1/6}),
\end{align*}
where we have used \eqref{e:int-f-l2} and \eqref{e:u-derivative}.
We also have
\begin{align*}
&\eta'(u^\eps)\!\!\int vf^\eps dv-\eta'(u)\!\!\int vfdv\\
=&(\eta'(u^\eps)-\eta'(u))\!\!\int vf^\eps dv+\left(\int vf^\eps dv-\int vfdv\right)\eta'(u),
\end{align*}
and
\begin{align*}
&u^\eps\eta'(u^\eps)\!\!\int f^\eps dv-u\eta'(u)\!\!\int fdv\\
=&(u^\eps\eta'(u^\eps)-u\eta'(u))\!\!\int f^\eps dv+\left(\int f^\eps dv-\int fdv\right)u\eta'(u).
\end{align*}
Using \eqref{e:fv-weak-limit}, \eqref{e:int-f-limit} and \eqref{e:u-weak-limit} we get the entropy inequality \eqref{e:entropy-inequality} by letting $\eps\rightarrow0.$

 Finally, combining with \eqref{e:f-limit-1}, \eqref{e:f-limit-2}, \eqref{e:u-limit},  we complete the proof of Theorem \ref{t:bounded-weak}.

%%%%%%%%%%%
\section*{Acknowledgments}
Huimin Yu's  research is supported in part by the National Natural Science Foundation of China
(Grant No.11671237, 11501333), China Scholarship Council No. 201708370075. Wentao Cao's research is supported by ERC Grant Agreement No. 724298.

\end{document}